\documentclass[12pt,a4paper,reqno]{amsart}

\usepackage{amsfonts}
\usepackage{amsmath}
\usepackage{amssymb}
\usepackage{amsthm}
\usepackage{amsxtra}
\usepackage{bbm}
\usepackage{braket}
\usepackage{comment}
\usepackage{extarrows}
\usepackage{graphicx}
\usepackage{latexsym}
\usepackage{mathrsfs}
\usepackage{yhmath}
\usepackage{nicematrix}
\usepackage{mathtools}

\usepackage{float}
\usepackage{booktabs}
\usepackage{verbatim}
\usepackage{xcolor}

\usepackage[pdfborder={0 0 0}]{hyperref} % hyperref must be loaded after natbib
\hypersetup{breaklinks, raiselinks}

\usepackage{bm}
\usepackage{tikz}
\usepackage[initials,lite]{amsrefs} % amsrefs must be loaded after hyperref
\AtBeginDocument{%
    \def\MR#1{}
}
\usepackage{cleveref}

\usepackage{a4wide}
\usepackage{fullpage}

\theoremstyle{plain}
\newtheorem{Theorem}{Theorem}[section]
\newtheorem{Lemma}[Theorem]{Lemma}
\newtheorem{Corollary}[Theorem]{Corollary}
\newtheorem{Proposition}[Theorem]{Proposition}

\theoremstyle{definition}

\newtheorem{Assumptions and Discussion}[Theorem]{Assumptions and Discussion}

\newtheorem{Example}[Theorem]{Example}

\newtheorem{Observations}[Theorem]{Observations}
\newtheorem{Notation}[Theorem]{Notation}

\theoremstyle{remark}

\newtheorem*{acknowledgment*}{Acknowledgment}

\usepackage[shortlabels]{enumitem}
\SetEnumerateShortLabel{a}{\textup{(\alph*)}}
\SetEnumerateShortLabel{A}{\textup{(\Alph*)}}
\SetEnumerateShortLabel{1}{\textup{(\arabic*)}}
\SetEnumerateShortLabel{i}{\textup{(\roman*)}}
\SetEnumerateShortLabel{I}{\textup{(\Roman*)}}
\def\lex{\operatorname{lex}}

\def\ceil#1{\left\lceil #1 \right\rceil}

\def\deg{\operatorname{deg}}
\def\dim{\operatorname{dim}}

\def\ini{\operatorname{in}} % initial ideal/term

\def\ker{\operatorname{ker}}
\def\KK{{\mathbb K}}
\def\lex{{\operatorname{lex}}}

\def\onto{\twoheadrightarrow}

\def\PP{{\mathbb P}}

\def\RR{{\mathbb R}}

\def\ZZ{{\mathbb Z}}

\newcommand\bda{{\bm a}}

\newcommand\bdH{{\bm H}}

\newcommand\bdn{{\bm n}}

\newcommand\bdT{{\bm T}}

\newcommand\bfM{\mathbf{M}}

\newcommand\bfX{\mathbf{X}}

\newcommand\calC{\mathcal{C}}

\newcommand\calF{\mathcal{F}}
\newcommand\calG{\mathcal{G}}

\newcommand\calS{\mathcal{S}}
\newcommand\calT{\mathcal{T}}

\newcommand\frakI{\mathfrak{I}}

\newcommand\frakL{\mathfrak{L}}

\newcommand\length{\operatorname{length}}

\def\r{\operatorname{r}}
\def\reg{\operatorname{reg}}

\begin{document}

\title{Fiber cones of rational normal scrolls are Cohen--Macaulay}%{Cohen--Macaulayness and Castelnuovo--Mumford regularity of the fiber cone of rational normal scrolls}

\author[Kuei-Nuan Lin, Yi-Huang Shen]{Kuei-Nuan Lin and Yi-Huang Shen}

%\thanks{\today}



\address{Department of Mathematics, The Penn State University, 
%Greater Allegheny Campus, 
McKeesport, PA, 15132, USA}
\email{kul20@psu.edu}

\address{CAS Wu Wen-Tsun Key Laboratory of Mathematics, School of Mathematical Sciences, University of Science and Technology of China, Hefei, Anhui, 230026, P.R.~China}
\email{yhshen@ustc.edu.cn}

\begin{abstract}
    In this short paper, we show that the fiber cones of rational normal scrolls are Cohen--Macaulay. As an application, we compute their Castelnuovo--Mumford regularities and $\bm{a}$-invariants, as well as the reduction number of the defining ideals of the rational normal scrolls. %In particular, these fiber cones are Cohen--Macaulay normal domain, and have rational singularities in characteristic zero and are $F$-rational in positive characteristic. 
    We also characterize the Gorensteinness of the fiber cone.
\end{abstract}

\maketitle

\section{Introduction}
In this short note, we show that the fiber cones of rational normal scrolls are Cohen--Macaulay. Let $R$ be a standard graded polynomial ring over some field $\KK$ and $I$ be an ideal of $R$ minimally generated by some forms $f_1,\ldots, f_s$ of the same degree. Those forms define a rational map $\phi$ whose image is a variety $X$. The bi-homogeneous coordinate ring of the graph of $\phi$ is the Rees algebra of the ideal $I$, and the homogeneous coordinate ring of the image is the fiber cone (special fiber ring) of this ideal. Here, we are mostly interested in these two types of blow-up algebras when the ideal $I$ describes the rational normal scroll.  

Rational normal scrolls are typical determinantal varieties,  central in the study of algebraic varieties. The study of determinantal varieties has attracted earnest attentions of algebraic geometers and  commutative algebraists, partly due to the beautiful structures involved and the interesting applications to the applied mathematics and statistics; see, for instance, \cite{BV}, \cite{MR3836659}, \cite{arXiv:2005.02909} and \cite{arXiv:2003.14232}, to name but a few. 

The study of Cohen--Macaulay property of the Rees algebra and the related fiber cone is an active research area. Most of the work is focusing on giving certain assumptions on the ground ring and the ideal;
%. One can deduce the Cohen--Macaulayness of the Rees algebra or its fiber cone, 
see for example, \cite{MR1723128}, \cite{MR0563540}, \cite{MR1128652}, and \cite{MR2431031}. However, not much is known about the Cohen--Macaulayness of these blow-up algebras of the determinantal varieties. The first well-known case is probably the maximal minors of a generic matrix where the fiber cone is the Grassmannian variety. Hochster \cite{MR314833} in 1973 showed that the Grassmannian is Cohen--Macaulay. In 1983, Eisenbud and Huneke \cite{MR0696134} showed that the posets defining the Rees algebra and fiber cone are wonderful, and obtained the Cohen--Macaulayness of those algebras. Recently, the first author of this paper and her coauthors proved the Cohen--Macaulayness of Rees algebras and its fibers cones of closed determinantal facet ideals in \cite{ALL} and sparse determinantal ideals in \cite{CDFGLPS}. 

All the work mentioned above came from generic matrices and their specializations. Blow-up algebras of rational normal scrolls were first studied by Conca, Herzog and Valla \cite{Sagbi} in 1996. And they proved that the Rees algebra and the fiber cone of a balanced rational normal scroll are Cohen--Macaulay normal domains. The difficulty of determining the Cohen--Macaulayness of those blow-up algebras is partly due to the lack of necessary information regarding the defining ideals of those algebras. Fortunately, Sammartano solved this implicitization problem in full generality for the rational normal scrolls 
%gave a complete description of the defining equations of the blow-up algebras of the rational normal scrolls 
in \cite{MR4068250} recently. And this opened the door of understanding those algebras for us.  

We need to be more precise now. In the following, let $n_1,\dots,n_d$ be a sequence of positive integers and $c=\sum_{i=1}^d n_i$. For each $i=1,\dots,d$, let $\calC_i\subseteq \PP^{c+d-1}$ be a rational normal curve of degree $n_i$ with complementary linear spans and let $\varphi_{i}:\PP^1\to \calC_i$ be the corresponding isomorphism. Whence, the related \emph{rational normal scroll} is simply
\[
    \calS_{n_1,\dots,n_d}\coloneqq \bigcup_{p\in \PP^1}\overline{\{\varphi_1(p),\varphi_2(p),\dots,\varphi_d(p)\}}\subseteq \PP^{c+d-1}.
\]
It is well-known (\cite{MR1416564}) that $\calS_{n_1,\dots,n_d}$ is uniquely determined by the sequence $n_1,\dots,n_d$ up to projective equivalence. And in suitable coordinates, the ideal $I_{n_1,\dots,n_d}$ of $\calS_{n_1,\dots,n_d}$ is generated by the maximal minors of the matrix
\begin{equation} 
    \bfX  \coloneqq  
    \scalebox{0.9}{$
        \begin{pNiceArray}{cccc|cccc|c|cccc} % use c instead of C in higher versions of nicematrix!
            x_{1,0} & x_{1,1} & \cdots & x_{1,n_1-1} &
            x_{2,0} & x_{2,1} & \cdots & x_{2,n_2-1} & \cdots &
            x_{d,0} & x_{d,1} & \cdots & x_{d,n_d-1} \\
            x_{1,1} & x_{1,2} & \cdots & x_{1,n_1} &
            x_{2,1} & x_{2,2} & \cdots & x_{2,n_2} & \cdots &
            x_{d,1} & x_{d,2} & \cdots & x_{d,n_d}
        \end{pNiceArray}.$} 
        \label{eqn:matrix-X}
\end{equation} 
In the following, we will write  $R=\KK[x_{i,j}\mid 0\le j\le n_i-1]$ for the ground ring, and $I=I_{n_1,\dots,n_d}=I_{2}(\bfX)$ for this ideal. Recall that the \emph{Rees algebra} of $I$ is $\mathcal{R}(I):=\oplus_{i\geq0}I^{i}t^{t}\subseteq R[t]$, and the related \emph{fiber cone} will be $\mathcal{F}(I)=\mathcal{R}(I)\otimes_R \mathbb{K}\cong\mathbb{K}[I]\subseteq R$, where $t$ is a new variable. As mentioned above, Sammartano gave the implicit equations of the defining ideals of these blow-up algebras in \cite{MR4068250}. Meanwhile, in the \emph{balanced case}, i.e., when $|n_{i}-n_{j}|\le 1$ for all $i,j$, Conca, Herzog, and Valla showed in \cite{Sagbi} that both $\mathcal{R}(I)$ and $\mathcal{F}(I)$ are Cohen--Macaulay normal domains. 
% Furthermore, if $\Char(\KK)=0$,
% then $\mathcal{R}(I)$ and $\mathcal{F}(I)$ have rational singularities.
% And if $\Char(\KK)>0$, then $\mathcal{R}(I)$ and $\mathcal{F}(I)$ are $F$-rational. 

In this paper, we will \emph{drop} the balanced assumption. We can still obtain the Cohen--Macaulayness of the fiber cone
$\mathcal{F}(I)$ in \Cref{thm:CM-fiber-cone}. By this, we generalize the corresponding result of Conca, Herzog, and Valla. As a quick application, we compute the Castelnuovo--Mumford regularity and $\bm{a}$-invariant of $\calF(I)$ in \Cref{thm:reg-2}. And the reduction number of the ideal $I$ is provided as well. As the second application, we characterize the Gorensteinness of the fiber cone.
% The Cohen--Macaulayness just established paves the way, since we have to resort to the balanced case. %by \cite[Theorem 3.7]{MR3341465} and \cite[Proposition 7.43]{MR2508056}. 

% Our main results are summarized in the following.
% \begin{Theorem}
%     \begin{enumerate}[1]
%         \item  Let $I_{n_1,\dots,n_d}$ be the defining ideal of the rational normal scroll $\calS_{n_1,\dots,n_d}\subseteq \PP^{c+d-1}$ for $c=n_1+\cdots+n_d$. Then the fiber cone $\calF(I_{n_1,\dots,n_d})$ is Cohen--Macaulay.
%         \item Closed formulas of the Castelnuovo--Mumford regularity and the $\bda$-invariant of the fiber cone $\calF(I_{n_1,\dots,n_d})$ as well as the reduction number of $I_{n_1,\dots,n_d}$ are given.
        
%     \end{enumerate}
% \end{Theorem}

\section{Cohen--Macaulayness of the fiber cones of rational normal scrolls} 

%In this section, we show in \Cref{thm:CM-fiber-cone} 
This section is devoted to prove
that the fiber cone associated with the rational normal
scroll is Cohen--Macaulay. To achieve that, we associate the initial ideal of the defining
ideal of the fiber cone 
with its Stanley--Reisner complex $\Delta$.
We will show that the ideal of the Alexander 
dual of this complex has linear quotients. Consequently, the original simplicial complex is shellable and Cohen--Macaulay. 
Hence, the fiber cone is Cohen--Macaulay as well. 

To establish the expected linear-quotients property, we have to resort to the binary-tree description of the facets of the simplicial complex $\Delta$ in \cite{MR4068250}.  To be more accurate, instead of just using the lexicographic order on the minimal monomial generators of the Alexander dual ideal, we need to group these generators first using the leaves of the binary trees; see \Cref{leafOrder}. 
The subsequent proof is also quite involved, since we have to analyze the binary trees in more detail and pin down those linear generators of the successive colon ideals.

Recall that a rational normal scroll is uniquely determined by a sequence of positive integers $\bdn=(n_1,...,n_d)$ with $n_1\le \cdots \le n_d$ by \cite{MR1416564}.
%In the following, let $c\coloneqq n_1+n_2+\cdots+n_d$. In suitable coordinates, the defining ideal $I=I_{n_1,\dots,n_d}$ of the corresponding rational normal scroll $\calS_{n_1, \ldots, n_d}\subseteq \PP^{c+d-1}$ is generated by the maximal minors of the matrix $\bfX$ where
%\begin{align*}
%    \bfX & \coloneqq (\bfX_1,\bfX_2,\dots,\bfX_d) \notag\\
%    & = 
%    \scalebox{0.95}{$
%        \begin{pNiceArray}{cccc|cccc|c|cccc}
%            x_{1,0} & x_{1,1} & \cdots & x_{1,n_1-1} &
%            x_{2,0} & x_{2,1} & \cdots & x_{2,n_2-1} & \cdots &
%            x_{d,0} & x_{d,1} & \cdots & x_{d,n_d-1} \\
%            x_{1,1} & x_{1,2} & \cdots & x_{1,n_1} &
%            x_{2,1} & x_{2,2} & \cdots & x_{2,n_2} & \cdots &
%            x_{d,1} & x_{d,2} & \cdots & x_{d,n_d}
%        \end{pNiceArray}.$} 
%    % \label{matrix-X}
%\end{align*}
%And here is the main result of this section.
Using the terminology stated in the above introduction, we are ready to present the main result of this paper.

\begin{Theorem}
    \label{thm:CM-fiber-cone}
    Let $I_{n_1,\dots,n_d}$ be the defining ideal of the rational normal scroll $\calS_{n_1,\dots,n_d}\subseteq \PP^{c+d-1}$ for $c=n_1+\cdots+n_d$. Then the fiber cone $\calF(I_{n_1,\dots,n_d})$ is Cohen--Macaulay.
\end{Theorem}

We need some preparations before laying out its proof.
Let us start by recalling the indispensable constructions in \cite{MR4068250} for the fiber cone $\calF(I)$ where $I=I_{n_1,\dots,n_d}$. First of all, a matrix $\bfM=\bfM_{n_1,\dots,n_d}$ was introduced in \cite[Section 2]{MR4068250} as follows. One starts with the first column of the $i$-th catalecticant block $\bfX_i$ of the original matrix $\bfX=(\bfX_1,\bfX_2,\dots,\bfX_d)$ for each $i$ increasingly in $i$, then the second column, and so on until one has used all columns except the last one for each block; when a block runs out of columns one simply skips it. The first $c-d$ columns of $\bfM$ form the submatrix 
\[
    \begin{pmatrix}
        x_{i(2),0} & x_{i(2)+1,0} & \cdots & x_{d,0} & x_{i(3),1} & x_{i(3)+1,1} & \cdots & x_{d,1} & x_{i(4),2} & \cdots & \cdots & x_{d,n_{d}-2}\\
        x_{i(2),1} & x_{i(2)+1,1} & \cdots & x_{d,0} & x_{i(3),2} & x_{i(3)+1,2} & \cdots & x_{d,1} & x_{i(4),3} & \cdots & \cdots & x_{d,n_{d}-1}
    \end{pmatrix},
\]
where the $i(l)$ denotes the least integer $i$ such that $n_i\ge l$.
The last $d$ columns of $\bfM$ are of the last column of the $i$-th block of $\bfX$ for each $i$, but ordered decreasingly in $i$:
\[
    \begin{pmatrix}
        x_{d,n_{d}-1} & x_{d-1,n_{d-1}-1} & \cdots & x_{2,n_{2}-1} & x_{1,n_{1}-1}\\
        x_{d,n_{d}} & x_{d-1,n_{d}-1} & \cdots & x_{2,n_{2}} & x_{1,n_{1}}
    \end{pmatrix}.
\]
Interested readers are invited to go through the examples in \cite[Example 2.1]{MR4068250}.

Let 
\[
    \pi: \KK[T_{\alpha,\beta}\mid 1\le \alpha<\beta \le c] \onto \calF(I)
\]
be the ring epimorphism determined by $T_{\alpha,\beta}\mapsto \det(\bfM_{\alpha,\beta})$, where $\bfM_{\alpha,\beta}$ is the $2\times 2$ submatrix of $\bfM$ using the $\alpha$-th and $\beta$-th columns. Then $P\coloneqq \ker(\pi)$ is called \emph{the defining ideal} of the fiber cone $\calF(I)$. Under suitable monomial ordering introduced in \cite{MR4068250}, the initial ideal $\ini(P)$ is squarefree and quadratic. One can then consider its associated Stanley--Reisner complex $\Delta$, which will be called the \emph{initial complex} of the fiber cone $\calF(I)$ following \cite{MR4068250}.

\begin{proof}
    [Proof of \Cref{thm:CM-fiber-cone}]
    When $c<d+4$, the fiber cone is the coordinate ring of the Grassmann variety $\mathbb{G}(1,c - 1) \subset P^{\binom{c}{2}-1}$; see \cite[Remark 3.15]{MR4068250}. Thus, the fiber cone is Cohen--Macaulay by \cite{MR314833}.
    
    Now, it remains to consider the case when $c\ge d+4$. Notice that $\calF(I)$ being Cohen--Macaulay is equivalent to saying that the defining ideal $P$ has this property. By \cite[Corollary 3.3.5]{MR2724673}, or more strongly, by \cite[Corollary 2.7]{arXiv1805.11923}, it suffices to show that $\ini(P)$ is Cohen--Macaulay. But this in turn is equivalent to showing that $\ini(P)^{\vee}$, the Alexander dual, has a linear resolution by \cite[Theorem 8.1.9]{MR2724673}. Because of this, we will show in \Cref{prop:linear-quotients} that $\ini(P)^{\vee}$ has linear quotients. This is sufficient for our purpose by \cite[Proposition 8.2.1]{MR2724673}. 
\end{proof}

In the following, we shall assume that $c\ge d+4$.
Before we really start proving the linear quotients property, we still need some preparations. Under the natural identification $T_{\alpha,\beta}\leftrightarrow (\alpha,\beta)$, the vertex set of the aforementioned simplicial complex $\Delta$ is
\[
    V = \Set{ (\alpha, \beta) | 1\le \alpha < \beta \le c}.
\]
Following \cite{MR4068250}, the vertices of $\Delta$ will also be described as open intervals in the real line $\RR$ with integral endpoints. And we will use the familiar notions of length, intersection, and containment of intervals.  

As usual, if $n$ is a positive integer, $[n]$ will be the set $\{1,2,\dots,n\}$. Then, with respect to the above matrix $\bfM$, for each $\alpha\in [c-d-2]$ and each $i\in [d]$, let $\gamma_{\alpha,i}$  be the least index $\gamma \geq \alpha+2$ such that the $\gamma$-th column of $\bfM$ involves variables from the block $\bfX_i$. Then, there exists some $\ell_{\alpha}$ with $2\le \ell_{\alpha}\le d+1$ such that
\begin{equation}
    \Set{\gamma_{\alpha, 1}, \dots, \gamma_{\alpha, d}} = \Set{ \alpha+2, \dots, \alpha +\ell_\alpha}\cup\Set{c-d+\ell_\alpha, \dots, c};
    \label{eqn:def-ell-alpha}
\end{equation}
c.f.~\cite[Definition 3.6]{MR4068250}. The key observation that we shall apply is the following result.

\begin{Lemma}
    [{\cite[Proposition 3.8]{MR4068250}}]
    \label{lem-facets}
    Suppose that $c\ge d+4$.
    A subset  $F \subseteq V$ is a facet of $\Delta$ if and only if the Hasse diagram $\calT_F$ of the poset $(F,\subseteq)$ satisfies the following conditions:
    \begin{enumerate}[i]
        \item $\calT_F$ is a rooted binary tree with the root $(1,c)$;

        \item there exists some $\alpha \in [c-d-2]$ such that the leaves of  $\calT_F$ are the intervals
            \begin{equation}
                \Set{ (\beta, \beta+1)  |  \beta \in \{\alpha, \ldots, \alpha + \ell_{\alpha}\} \cup \{c-d+\ell_{\alpha}-1, \ldots, c-1\}};
                \label{leaves-alpha}
            \end{equation}

        \item \label{lem-facets-iii} if  $\frakI\in F$ is a node of  $\calT_F$ with only one child $\frakI_1$, then $\length(\frakI) = \length(\frakI_1) +1$ and the unique unitary interval in $\frakI\setminus \frakI_1$ does not belong to $F$;

        \item \label{lem-facets-iv} if  $\frakI\in F$ is a node of  $\calT_F$ with two children  $\frakI_1,  \frakI_2$, then $\frakI_1\cap  \frakI_2= \varnothing$ and $ \length(\frakI_1) +\length(\frakI_2)=\length(\frakI) $.
    \end{enumerate}
    In this case we have  $|F| = c+d$.
\end{Lemma}

In the following, when we say the node $(\alpha',\beta')$ is a \emph{left child} of the node $(\alpha'',\beta'')$ in the binary tree $\calT_F$, we mean it is a child and $\alpha'=\alpha''$. The notions of \emph{right child}, \emph{left sibling} and \emph{right sibling} can be similarly defined.

We will call the set in \eqref{leaves-alpha} as the \emph{leaves set} of $F$ and denote it by $\frakL_{\alpha}$.

Notice that it is uniquely determined by its smallest point $\alpha\in[c-d-2]$ for the given matrix $\bfM$. The following easy fact will be needed later.

\begin{Lemma}
    \label{leaves-set-diff-1}
    For each $\alpha \in [c-d-3]$, the cardinalities $|\frakL_{\alpha}\setminus \frakL_{\alpha+1}|=|\frakL_{\alpha+1}\setminus \frakL_{\alpha}|=1$.
\end{Lemma}

\begin{proof}
    Recall that for each such $\alpha$ and each $i\in[d]$, the integer $\gamma_{\alpha,i}$ is the least index $\gamma\ge \alpha+2$ such that the $\gamma$-th column of $\bfM$ involves variables from the block $\bfX_i$. Furthermore, the union in (\ref{eqn:def-ell-alpha}) is actually disjoint.
   
    Therefore, $\gamma_{\alpha+1,i}\ge \gamma_{\alpha,i}$. And when $\gamma_{\alpha,i}\ge (\alpha+1)+2$, then $\gamma_{\alpha+1,i}=\gamma_{\alpha,i}$.  Now, suppose that $\alpha+2=\gamma_{\alpha,i_0}$. Then for $i\ne i_0$, we have $\gamma_{\alpha,i}=\gamma_{\alpha+1,i}$. 
    Whence, $\frakL_{\alpha+1}\setminus \frakL_{\alpha} = \{(\gamma_{\alpha+1,i_0},\gamma_{\alpha+1,i_0}+1)\}$ and $\frakL_{\alpha} \setminus \frakL_{\alpha+1}=\{(\alpha,\alpha+1)\}$.
\end{proof}

It is time to introduce the expected linear ordering.
Let $S=\KK[T_v\mid v\in V]$ be the base ring. Endow it with the lexicographic order $>_{\lex}$ with respect to the linear order on the variables $T_v$'s such that
\[
    T_{v_1}>T_{v_2} \quad \Leftrightarrow \quad \text{the leftmost nonzero component of $v_1-v_2$ is negative};
\]
here, by abuse of notation, we also think of $V$ as a set of integral points on $\ZZ_+^2$. Recall that the facet $F$ of the initial complex $\Delta$ is bijectively related to the monomial 
\[
    \widehat{\bdT^{F}} \coloneqq \prod_{v\in V\setminus F}T_v 
\]
in the minimal monomial generating set $G((\ini(P))^{\vee})$ by \cite[Lemma 1.5.3]{MR2724673}. Therefore, to impose a linear order on $G((\ini(P))^{\vee})$ amounts to giving a linear order of the facets. For later reference, facets with the common leaves set given by \eqref{leaves-alpha} will be grouped into the family \emph{$\calG_{\alpha}$}. 

\begin{Notation}\label{leafOrder}
    The expected linear order on the facets, denoted by $\succ$, will be as follows.
    \begin{enumerate}[a]
        \item If two facets $F\in \calG_{\alpha}$ and $F'\in \calG_{\alpha'}$ respectively with $\alpha>\alpha'$, then $F\succ F'$.
        \item Next, consider the facets $F$ and $F'$ within the same group $\calG_{\alpha}$. Then $F\succ F'$ if and only if $\widehat{\bdT^{F}} >_{\lex} \widehat{\bdT^{F'}}$.
    \end{enumerate}
\end{Notation}

\begin{Example}
    \label{first-facet}
    The first facet of the group $\calG_{\alpha}$ with respect to our ordering $\succ$ is the one with vertices
    \begin{align*}
        (1,c),(2,c),\dots, (c-2,c)
    \end{align*}
    in addition to the leaves given in \eqref{leaves-alpha}.

    For example, let $\bdn=(2,2,4,4)$. Then the matrix
    \[
        \bfM=
        \begin{pmatrix}
            x_{1,0} & x_{2,0} & x_{3,0} & x_{4,0} & x_{3,1} & x_{4,1} & x_{3,2} & x_{4,2} & x_{4,3} & x_{3,3} & x_{2,1} & x_{1,1} \\
            x_{1,1} & x_{2,1} & x_{3,1} & x_{4,1} & x_{3,2} & x_{4,2} & x_{3,3} & x_{4,3} & x_{4,4} & x_{3,4} & x_{2,2} & x_{1,2}
        \end{pmatrix}.
    \]
    If we choose $\alpha=2$, then the leaves set is
    \[
        \Set{(2,3),(3,4), (4,5),(5,6),(10,11),(11,12)};
    \]
    see also \cite[Example 3.9]{MR4068250}. The collection of open intervals in the \Cref{fig:facets-2244} is the first facet with respect to the given $\alpha=2$ in the corresponding complex $\Delta$.
    
    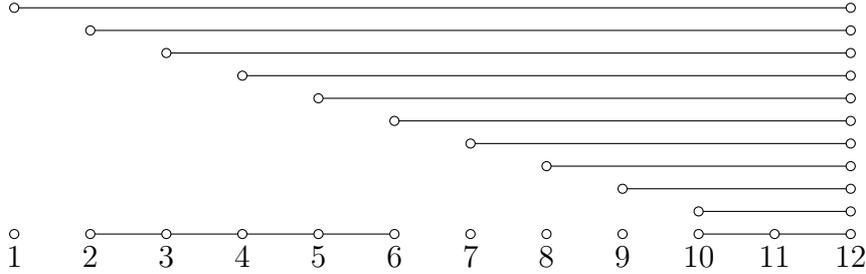
\begin{figure}[htb]
        \begin{tikzpicture}
            \draw (1,-1.2)--(5,-1.2);
            \draw (9,-1.2)--(11,-1.2);
            \draw [fill=white]  (0,-1.2) circle [radius=0.06];
            \draw [fill=white]  (1,-1.2) circle [radius=0.06];
            \draw [fill=white]  (2,-1.2) circle [radius=0.06];
            \draw [fill=white]  (3,-1.2) circle [radius=0.06];
            \draw [fill=white]  (4,-1.2) circle [radius=0.06];
            \draw [fill=white]  (5,-1.2) circle [radius=0.06];
            \draw [fill=white]  (6,-1.2) circle [radius=0.06];
            \draw [fill=white]  (7,-1.2) circle [radius=0.06];
            \draw [fill=white]  (8,-1.2) circle [radius=0.06];
            \draw [fill=white]  (9,-1.2) circle [radius=0.06];
            \draw [fill=white] (10,-1.2) circle [radius=0.06];
            \draw [fill=white] (11,-1.2) circle [radius=0.06];

            \node at  (0,-1.5) {1};
            \node at  (1,-1.5) {2};
            \node at  (2,-1.5) {3};
            \node at  (3,-1.5) {4};
            \node at  (4,-1.5) {5};
            \node at  (5,-1.5) {6};
            \node at  (6,-1.5) {7};
            \node at  (7,-1.5) {8};
            \node at  (8,-1.5) {9};
            \node at  (9,-1.5) {10};
            \node at (10,-1.5) {11};
            \node at (11,-1.5) {12};

            \draw (9,-0.9)--(11,-0.9);
            \draw [fill=white] (9,-0.9) circle [radius=0.06];
            \draw [fill=white] (11,-0.9) circle [radius=0.06];

            \draw (8,-0.6)--(11,-0.6);
            \draw [fill=white] (8,-0.6) circle [radius=0.06];
            \draw [fill=white] (11,-0.6) circle [radius=0.06];

            \draw (7,-0.3)--(11,-0.3);
            \draw [fill=white] (7,-0.3) circle [radius=0.06];
            \draw [fill=white] (11,-0.3) circle [radius=0.06];

            \draw (6,0.0)--(11,0.0);
            \draw [fill=white] (6,0.0) circle [radius=0.06];
            \draw [fill=white] (11,0.0) circle [radius=0.06];

            \draw (5,0.3)--(11,0.3);
            \draw [fill=white] (5,0.3) circle [radius=0.06];
            \draw [fill=white] (11,0.3) circle [radius=0.06];

            \draw (4,0.6)--(11,0.6);
            \draw [fill=white] (4,0.6) circle [radius=0.06];
            \draw [fill=white] (11,0.6) circle [radius=0.06];

            \draw (3,0.9)--(11,0.9);
            \draw [fill=white] (3,0.9) circle [radius=0.06];
            \draw [fill=white] (11,0.9) circle [radius=0.06];

            \draw (2,1.2)--(11,1.2);
            \draw [fill=white] (2,1.2) circle [radius=0.06];
            \draw [fill=white] (11,1.2) circle [radius=0.06];

            \draw (1,1.5)--(11,1.5);
            \draw [fill=white] (1,1.5) circle [radius=0.06];
            \draw [fill=white] (11,1.5) circle [radius=0.06];

            \draw (0,1.8)--(11,1.8);
            \draw [fill=white] (0,1.8) circle [radius=0.06];
            \draw [fill=white] (11,1.8) circle [radius=0.06];
        \end{tikzpicture}
        \caption{The first facet in $\calG_\alpha$ when $\alpha=2$ and $\bdn=(2,2,4,4)$}
        \label{fig:facets-2244}
    \end{figure}
    
    And the Hasse diagram of the poset as considered in \Cref{lem-facets} is pictured in \Cref{fig:tree-2244}.
    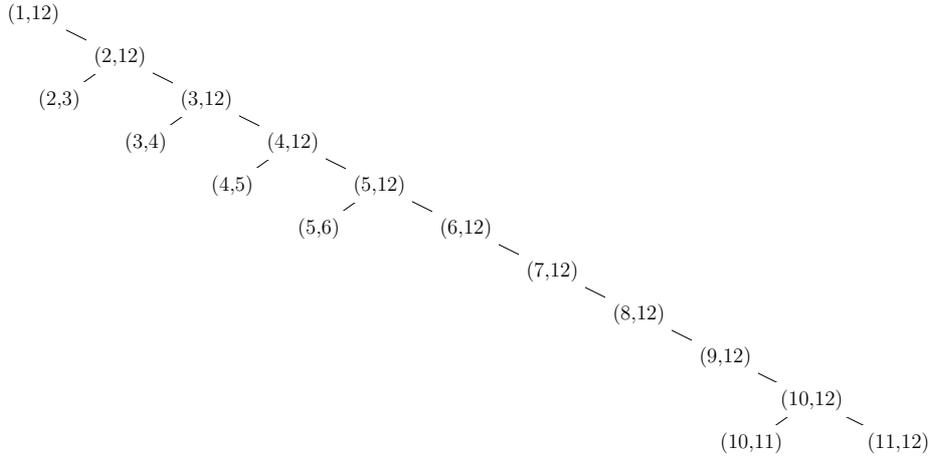
\begin{figure}[htb]
        \scalebox{0.65}{
            \begin{tikzpicture}[scale=1.75]
                \draw (1,-0.5)--(0,0);
                \draw (1,-0.5)--(0.3,-1);
                \draw (1,-0.5)--(2,-1);
                \draw (1.3,-1.5)--(2,-1);
                \draw (3,-1.5)--(2,-1);
                \draw (3,-1.5)--(2.3,-2);
                \draw (3,-1.5)--(4,-2);
                \draw (4,-2.0)--(3.3,-2.5);
                \draw (4,-2.0)--(5,-2.5);
                \draw (5,-2.5)--(6,-3.0);
                \draw (6,-3.0)--(7,-3.5);
                \draw (7,-3.5)--(8,-4.0);
                \draw (8,-4.0)--(9,-4.5);
                \draw (9,-4.5)--(10,-5.0);
                \draw (9,-4.5)--(8.3,-5.0);
                \node[fill=white] at (0,0) {(1,12)};
                \node[fill=white] at (0.3,-1.0) {(2,3)};
                \node[fill=white] at (1,-0.5) {(2,12)};
                \node[fill=white] at (1.3,-1.5) {(3,4)};
                \node[fill=white] at (10,-5.0) {(11,12)};
                \node[fill=white] at (2,-1.0) {(3,12)};
                \node[fill=white] at (2.3,-2.0) {(4,5)};
                \node[fill=white] at (3,-1.5) {(4,12)};
                \node[fill=white] at (3.3,-2.5) {(5,6)};
                \node[fill=white] at (4,-2.0) {(5,12)};
                \node[fill=white] at (5,-2.5) {(6,12)};
                \node[fill=white] at (6,-3.0) {(7,12)};
                \node[fill=white] at (7,-3.5) {(8,12)};
                \node[fill=white] at (8,-4.0) {(9,12)};
                \node[fill=white] at (8.3,-5.0) {(10,11)};
                \node[fill=white] at (9,-4.5) {(10,12)};
            \end{tikzpicture}
        }
        \caption{Binary tree of the first facet in $\calG_\alpha$ when $\alpha=2$ and $\bdn=(2,2,4,4)$}
        \label{fig:tree-2244}
    \end{figure}
\end{Example}

Now, we are ready to show that the Alexander dual ideal $(\ini(P))^{\vee}$ has linear quotients.
\begin{Proposition}
    \label{prop:linear-quotients}
   When $c\ge d+4$, the ideal $(\ini(P))^{\vee}$ has linear quotients with respect to the given ordering $\succ$.
\end{Proposition}

\begin{proof}
    Let $F$ be a facet of $\Delta$ and suppose that $F\in \calG_{\alpha}$. We want to show that the minimal monomial generating set \emph{$G_F$}$\coloneqq G(I_F)$ of the colon ideal
    \[
        I_F\coloneqq \braket{\widehat{\bdT^{F'}}\mid F'\succ F}: \widehat{\bdT^F}
    \]
    is linear. Notice that 
    \[
        \braket{\widehat{\bdT^{F'}}}:\widehat{\bdT^{F}}=\bdT^{F\setminus F'}\coloneqq \prod_{v\in F\setminus F'}T_v.
    \]

    Suppose that $F$ is the first facet within $\calG_{\alpha}$ with respect to the designated linear order $\succ$. If $\alpha=c-d-2$, then $F$ is actually the first facet of $\Delta$ and there is nothing to show in this case. Otherwise, $\alpha<c-d-2$. Let $F'$ be the first facet of $\calG_{\alpha+1}$. It is clear that $F'\succ F$ and
    \[
        F\setminus F'=\frakL_{\alpha}\setminus \frakL_{\alpha+1}=\{({\alpha},{\alpha}+1)\}
    \]
    by \Cref{leaves-set-diff-1} and  \Cref{first-facet}.  Notice that for any facet $F''$ preceding $F$, $({\alpha},{\alpha}+1)\in F\setminus F''$.  Whence, the colon ideal $I_F$ is principal and linear.

    Suppose instead that $F$ is not the first facet in $\calG_{\alpha}$. This situation is more involved. We will break the long proof into three parts, according to the following plans.
    \begin{enumerate}
        \item Firstly, we describe the expected linear generating set {$G_F$}. 
        \item Next, we show that the variables described above really show up as some colons of monomials, and they are the only variables possessing this property.
        \item Finally, we show that we don't have minimal monomial generators in $G_F$ of higher degrees.
    \end{enumerate}
    After showing the above three parts, it is clear then that  the ideal $(\ini(P))^{\vee}$ has linear quotients. At this point, one may suggest only proving the last step. But to successfully give the proof, we have to clear higher-degree generators by the linear ones.

    \subsection{Linear generators of the colon ideal $I_F$}
    \label{Linear generators}
    We separate the vertices in $F$ into collections for each fixed facet $F\in \calG_{\alpha}$. For $b\in [c-1]$, let
    \begin{equation}
        \label{F_b}
        F_b\coloneqq \Set{v\in F\mid v=(b,*)}.
    \end{equation}
    If the cardinality \emph{$k_b$}$\coloneqq |F_b|\ge 1$, we may assume that
    \[
        F_b=\Set{(b,i_{b,1}),(b,i_{b,2}),\dots,(b,i_{b,k_b})}
    \]
    with $i_{b,1}<i_{b,2}<\cdots<i_{b,k_b}$. 
    It is clear that $(b,i_{b,j+1})$ is the parent node of $(b,i_{b,j})$ when $1\le j< k_b$.
    Now, assume that $k_b\ge 1$, and we describe whether $T_{b,i_{b,j}}\in G_F$ for $j\in [k_b]$.
    \begin{enumerate}[a]
        \item If $j=k_b$, then $T_{b,i_{b,k_b}}\notin G_F$.
        \item Suppose that $2\le j< k_b$. If
            \begin{enumerate}[i]
                \item the node at $(b,i_{b,j})$ has a right sibling in $\calT_F$, or
                \item the node at $(b,i_{b,j})$ has two children in $\calT_F$,
            \end{enumerate}
            then $T_{b,i_{b,j}}\in G_F$. Otherwise, $T_{b,i_{b,j}}\notin G_F$.
        \item Suppose that $j=1<k_b$.
            \begin{enumerate}[i]
                \item Suppose that $(b,i_{b,1})$ is a leaf in $\calT_F$, namely that $i_{b,1}=b+1$.
                    \begin{enumerate}[1]
                        \item If $b\ne \alpha$, then $T_{b,i_{b,1}}\notin G_F$.
                        \item Suppose that $b=\alpha$, i.e., $(b,i_{b,1})$ is the leftmost leaf in $\calT_F$.
                            \begin{itemize}
                                \item If $b=c-d-2$, then $\calG_{\alpha}$ is the first group to be considered. Whence, we expect $T_{b,i_{b,1}}\notin G_F$.
                                \item Otherwise, we expect $T_{b,i_{b,1}}=T_{\alpha,\alpha+1} \in G_F$ and $\alpha\neq c-d+2$.
                            \end{itemize}
                    \end{enumerate}
                \item Otherwise, $i_{b,1}>b+1$. Whence, we always have $T_{b,i_{b,1}}\in G_F$.
            \end{enumerate}
    \end{enumerate}
    For later reference, we denote the set of the expected linear generators of $F$ with respect to $b$ described above as $LG_{F,b}$. Then in the next two subsections, we will see that the minimal generating set $G_F$ is precisely the set
    \[
        LG_F\coloneqq \bigcup_{b\in [c-1]} LG_{F,b}.
    \]

    \begin{Example}
        Consider the case when $\bdn=(n_1,n_2,n_3)=(2,4,5)$. It is not difficult to see that 
        \begin{align*}
            F=\{(1,2),\, (1,3),\, &(1,4),\, (1,6), \,(1,7),\, (1,8), \,(1,9), \,(1,11),\,\\ 
            &(2,3), \,(3,4), \,(4,5),\, (4,6), \,(9,11), \,(10,11)\}
        \end{align*}
        is a facet of the associated complex $\Delta$. Consider $b=1$. The subset $F_b$ consists of the vertices written on the first line. And our criterion above says that the corresponding linear generators for $b=1$ are
        \[
            LG_{F,1}=\{T_{1,2}, T_{1,3}, T_{1,4}, T_{1,6}, T_{1,9}\}\subseteq LG_F.
        \]
        Indeed, it is not difficult to see that equality holds in this special case.
    \end{Example}

    \subsection{Why the vertex of $F$ shows or not shows up in $G_F$?}
    Before we proceed, notice first the following consequences of \Cref{lem-facets}.
    \begin{Observations}
        \label{rmk-key-facts}
        Let $F$ be a facet of the aforementioned initial complex $\Delta$. Recall that it is convenient to consider the vertex $(\alpha,\beta)$ in $\Delta$ as an open interval in the real line $\RR$ with integral endpoints.
        \begin{enumerate}[i]
            \item Vertices of $F$ are \emph{non-crossing}, i.e., if we consider the vertices $\frakI_1,\frakI_2$ of $F$ as open intervals, then $\frakI_1\cap \frakI_2=\varnothing$, or $\frakI_1\subseteq \frakI_2$, or $\frakI_2\subseteq \frakI_1$.
            \item \label{rmk-key-facts-2} 
                If $(\alpha,\beta)\in F$ is not the root and has no sibling, then its parent node is either $(\alpha-1,\beta)$ or $(\alpha,\beta+1)$. Correspondingly, $(\alpha-1,\alpha)$ or $(\beta,\beta+1)$ is not a leaf of $\calT_F$.
            \item \label{rmk-key-facts-3} 
                If as an open interval, $(\alpha,\beta)$ contains no leaf of $\calT_F$, then the vertex $(\alpha,\beta)$ does not belong to $F$.
            \item \label{rmk-key-facts-4}  If both $(\alpha,\beta)$ and  $(\alpha,\beta')$ belong to $F$ with $\beta<\beta'$ such that $(\beta,\beta')$ contains no leaf of $\calT_F$, then $(\alpha,\beta+j)\in F$ for all $j$ with $0 < j < \beta'-\beta$. To see this, notice that for each such fixed $j$, 
                $F$ must have at least one vertex of the form $(\beta+j,\gamma)$ or $(\alpha',\beta+j)$.
                In the first subcase, $\gamma\le \beta'$ is impossible by the previous observation. But when $\gamma>\beta'$, this is also impossible, since the two open intervals $(\beta+j,\gamma)$ and $(\alpha,\beta')$ cross. As for the second subcase,  by the previous observation, it is clear that $\alpha'< \beta$. If $\alpha<\alpha' <\beta$, then the two open intervals $(\alpha',\beta+j)$ and $(\alpha,\beta)$ cross. And if $\alpha'<\alpha$, then the two open intervals $(\alpha',\beta+j)$ and $(\alpha,\beta')$ cross. Therefore, $\alpha'=\alpha$ and $(\alpha,\beta+j)$ is a vertex of $F$, as expected.
           
            \item \label{rmk-key-facts-6} Suppose that $(\gamma,\gamma+1)\in V$ is not a leaf of $\calT_F$. Consider the smallest interval (node) in $\calT_F$ that contains $(\gamma,\gamma+1)$. Then by \Cref{lem-facets} and the minimality, this node is either $(\gamma',\gamma+1)$ or $(\gamma, \gamma')$ for some $\gamma'$. And correspondingly it has a unique child node $(\gamma',\gamma)$ or $(\gamma+1,\gamma')$.
        \end{enumerate}
        Another fact is that the variable $T_v\in G_F$ if and only if $F\setminus F'=\{v\}$ for some
        $F'\succ F$. And for this pair $F$ and $F'$, if $F\setminus F'=\{(i,j)\}$ while $F'\setminus F=\{(i',j')\}$, then
        \begin{equation*}
            \text{either $i'>i$ or $i'=i$ and $j'>j$.} 
        \end{equation*}
    \end{Observations}

    Now, we are ready to justify the linear generating set $LG_F$ presented in the previous subsection \ref{Linear generators}. Correspondingly, we also have three cases.
    \begin{enumerate}[a]
        \item If $j=k_b$, we show that $T_{b,i_{b,k_b}}\notin G_F$. We don't have to worry about the case when $b=1$, since $(1,c)$ is the common root.  Suppose now for contradiction that there is some $F'$ with $F'\succ F$ and $F\setminus F'=\{(b,i_{b,k_b})\}$. Thus, we have two subcases.

            Suppose first that $F\in \calG_{\alpha}$ and $F'\in \calG_{\alpha'}$ with $\alpha<\alpha'$. Since necessarily $(\alpha,\alpha+1)\in \frakL_{\alpha}\setminus \frakL_{\alpha'}\subseteq F\setminus F'$ while $|F\setminus F'|=1$, it follows that $b=\alpha$ and $i_{b,k_b}=\alpha+1$. 
            Since the leftmost leaf of $\calT_{F}$ is $(\alpha,\alpha+1)$ and $|F_\alpha|=1$, $F$ must contain the vertex $(\alpha-1,\alpha+1)$
            by item \ref{rmk-key-facts-2} of \Cref{rmk-key-facts}.
            Hence, $F'$ must also contain this vertex. But the leftmost leaf of $\calT_{F'}$ is $(\alpha',\alpha'+1)$ with $\alpha'>\alpha$, contradicting the item \ref{rmk-key-facts-3} of \Cref{rmk-key-facts}. 

            Suppose now instead that $F$ and $F'$ have the same leaves set, i.e., $\alpha=\alpha'$. Then $(b,i_{b,k_b})$ is not a leaf of $\calT_F$ and must have some child. The parent node of $(b,i_{b,k_b})$ in $\calT_F$ must have the form $(b',i_{b,k_b})$ for some $b'<b$. 
            Notice that the unique vertex in $F'\setminus F$ that the vertex $(b,i_{b,k_b})$ is switched to cannot take the form $(b,b'')$ with $b''>i_{b,k_b}$, since the interval $(b,b'')$ will cross the interval $(b',i_{b,k_b})$. Meanwhile, this unique vertex cannot take the form $(b'',b''')$ with $b''>b$, since $(b',i_{b,k_b})$ will not have a legal branch in $\calT_{F'}$; see \ref{lem-facets-iii} and \ref{lem-facets-iv} of \Cref{lem-facets}. Therefore, we have a contradiction to the fact stated at the end of \Cref{rmk-key-facts}.

            Thus, $T_{b,i_{b,k_b}}\notin G_F$.

        \item Suppose that $2\le j< k_b$.
            \begin{enumerate}[i]
                \item Suppose that the node $(b,i_{b,j})$ has a sibling in $\calT_F$. It is clear that the parent node of $(b,i_{b,j})$ is $(b,i_{b,j+1})$ and $(b,i_{b,j})$ is the left child in the binary tree. Whence, the sibling of $(b,i_{b,j})$ is $(i_{b,j},i_{b,j+1})$. Now, let
                    \[
                        F'\coloneqq F\cup\{(i_{b,j-1},i_{b,j+1})\} \setminus \{(b,i_{b,j})\}.
                    \]
                    Obviously $F'\succ F$. And it is easy to verify that $F'$ is a legal facet and the node $(b,i_{b,j+1})$ has two children $(b,i_{b,j-1})$ and $(i_{b,j-1},i_{b,j+1})$ in $\calT_{F'}$. This implies that $T_{b,i_{b,j}}\in G_F$.

                \item Suppose that the node $(b,i_{b,j})$ has two children in $\calT_F$. They have to be $(b,i_{b,j-1})$ and $(i_{b,j-1},i_{b,j})$. We similarly define
                    \[
                        F'\coloneqq F\cup\{(i_{b,j-1},i_{b,j+1})\} \setminus \{(b,i_{b,j})\},
                    \]
                    and see that $T_{b,i_{b,j}}\in G_F$.
                \item Suppose that the node $(b,i_{b,j})$ has only one child and has no right sibling. 
                    %Surely $(b,i_{b,j})$ is not a leaf. Whence, there is no leaf within $(i_{b,j-1},i_{b,j+1})$. It follows that 
                    Whence, $i_{b,j\pm 1}=i_{b,j}\pm 1$ respectively and $(i_{b,j-1},i_{b,j+1})$ contains no leaf in $\calT_F$ by \Cref{rmk-key-facts} \ref{rmk-key-facts-2}. Suppose for contradiction that we have some facet $F'$ such that $F'\succ F$ and $F\setminus F'=\{(b,i_{b,j})\}$. Since $(b,i_{b,j})$ is not a leaf, $F$ and $F'$ have identical leaves set. Therefore $(i_{b,j-1},i_{b,j+1})$ contains no leaf in $\calT_{F'}$ as well. The node $(b,i_{b,j-1})$ must be a grandchildren of $(b,i_{b,j+1})$ in $\calT_{F'}$ by \Cref{rmk-key-facts} \ref{rmk-key-facts-4} and \Cref{lem-facets} \ref{lem-facets-iii} and \ref{lem-facets-iv}. Meanwhile, $(b,i_{b,j})$ is the only possible option to connect $(b,i_{b,j-1})$ with $(b,i_{b,j+1})$. This makes $F=F'$, a contradiction. Therefore, $T_{b,i_{b,j}}\notin G_F$.
            \end{enumerate}
        \item Suppose that $j=1<k_b$.
            \begin{enumerate}[i]
                \item We first consider the case when $(b,i_{b,1})$ is a leaf in $\calT_F$, namely that $i_{b,1}=b+1$. Suppose that $T_{b,i_{b,1}}\in G_T$. Then we have $F\in \calG_{\alpha}$ and $F'\in \calG_{\alpha'}$ for some $\alpha<\alpha'$ such that $F\setminus F'=\{(b,b+1)\}$. It is clear then that $\alpha<\alpha'\le c-d-2$. And again, since $(\alpha,\alpha+1)\in \frakL_{\alpha}\setminus \frakL_{\alpha'}\subseteq F\setminus F'$ while $|F\setminus F'|=1$, it follows that $b=\alpha$. 

                    Conversely, suppose that $b=\alpha<c-d-2$ with $F\in \calG_\alpha$. Notice that $\frakL_{\alpha+1}\setminus \frakL_{\alpha}=\{(\beta,\beta+1)\}$ for some $\beta$ by \Cref{leaves-set-diff-1}, and this $(\beta,\beta+1)\notin F$.
                    Meanwhile, since the invariant $\ell_\alpha$ defined before \eqref{eqn:def-ell-alpha} satisfies $\ell_{\alpha}\ge 2$, both $(\alpha,\alpha+1)$ and $(\alpha+1,\alpha+2)$ are leaves of $\calT_F$.
                    Thus, as $k_{\alpha}\ge 2$, the node $(\alpha,\alpha+1)$ has a right sibling $(\alpha+1,i_{\alpha,2})$ by \Cref{lem-facets} \ref{lem-facets-iii}. It is clear that any vertex in $\calT_F$ strictly containing $(\alpha,\alpha+1)$ will contain its parent node $(\alpha,i_{\alpha,2})$, and consequently will contain $(\alpha+1,i_{\alpha,2})$. Now, it is not difficult to verify that
                    \[
                        F'\coloneqq F\cup\{(\beta,\beta+1)\}\setminus \{(\alpha,\alpha+1)\} \in \calG_{\alpha+1}
                    \]
                    is a well-defined facet, and therefore, $T_{\alpha,\alpha+1}\in G_F$.

                \item Next, we assume that $(b,i_{b,1})$ is not a leaf in $\calT_F$, namely that $i_{b,1}>b+1$. Since $j=1$, $(b,b+1)$ is not a leaf and $(b,i_{b,1})$ has only the right child $(b+1,i_{b,1})$ in $\calT_F$. We have two subcases.
                    \begin{itemize}
                        \item Suppose that $(b,i_{b,1})$ has a right sibling in $\calT_F$. It has to be $(i_{b,1},i_{b,2})$.
                        \item Suppose that $(b,i_{b,1})$ does not have a right sibling in $\calT_F$. 
                            %Whence, $(i_{b,1},i_{b,2})$ does not contain any leaf of $\calT_F$. It follows that 
                            Whence, $i_{b,2}=i_{b,1}+1$ and $(i_{b,1},i_{b,1}+1)$ is not a leaf of $\calT_F$ by \Cref{rmk-key-facts} \ref{rmk-key-facts-2}.
                    \end{itemize}
                    In either subcase, let
                    \[
                        F'\coloneqq F\cup\{(b+1,i_{b,2})\} \setminus \{(b,i_{b,1})\},
                    \]
                    and we can see as above that $T_{b,i_{b,1}}\in G_F$. And this completes our argument in this subsection.
            \end{enumerate}
    \end{enumerate}

    \subsection{No minimal monomial generator of higher degrees}

    In this subsection, we finish the proof by showing that $G_F=LG_F$.

    Suppose for contradiction that $F\in \calG_{\alpha}$ and $F'\in \calG_{\alpha'}$ with $F'\succ F$ such that $\braket{\widehat{\bdT^{F'}}}:\widehat{\bdT^{F}}$ provides a minimal monomial generator of higher degree in $G_F\setminus LG_F$.
    Whence, none of the linear generators in $LG_F$ will ever divide the principal generator of $\braket{\widehat{\bdT^{F'}}}:\widehat{\bdT^{F}}$. If we apply the identification $T_{i,j}\leftrightarrow (i,j)\in \Delta$, then this is just 
    \begin{equation}
        \text{$LG_F\cap (F\setminus F')=\varnothing$, or more simply, $LG_F\subseteq F\cap F'$.}
        \label{eqn:high-deg}
    \end{equation}

    If $\alpha\ne \alpha'$, let $a\le \alpha$ be the smallest such that $(a,\alpha+1)\in F$. 
    We claim that $(a,\alpha+1)\in LG_F$. It is clear when $a=\alpha$ since $\alpha<\alpha'\le c-d-2$. Thus, we will assume $a<\alpha$. 
    The parent node of $(a,\alpha+1)$ has the form either $(a',\alpha+1)$ for some $a'<a$, or $(a,\alpha'')$ for some $\alpha''>\alpha+1$.
    The first case contradicts the minimality of $a$. Thus, we have the second case and consequently $|F_a|\ge 2$.
   
    Whence, $(a,\alpha+1)\in LG_F$, still establishing the claim. Meanwhile, since $\alpha'>\alpha$, $(a,\alpha+1)$ does not contain any leaf of $\calT_{F'}$. This implies that $(a,\alpha+1)\in F\setminus F'$, contradicting the assumption in \eqref{eqn:high-deg}.

    Therefore, we will assume in the following that $\alpha=\alpha'$, i.e., $\calT_F$ and $\calT_{F'}$ have a common leaf set. Now, let $b\in [c-1]$ be the smallest such that $F_b\ne F_b'$. Here, $F_b'$ is a subset of $F'$, just as $F_b$ defined for $F$ in \eqref{F_b}.
    \begin{enumerate}[a]
        \item Suppose that $k_b=1$. By the minimality of $b$ and the fact that $\widehat{\bdT^{F'}} >_{\lex} \widehat{\bdT^{F}}$, 
        we have $(b,j)\notin F_b'$ for all $j$ with $b+1\le j\le i_{b,1}$.
        Therefore, either $(b,i')\in F_b'$ for some $i'>i_{b,1}$ or $F_b'=\varnothing$ .

            In the first subcase, notice that the parent node of $(b,i_{b,1})$ in $\calT_F$ must take the form $(b',i_{b,1})$ for some $b'<b$. By the minimality of $b$, $(b',i_{b,1})\in F'$. But since $(b,i')$ will cross $(b',i_{b,1})$ when considered as open intervals, this is impossible.

            In the remaining subcase, one can conclude that $F'$ contains some vertices of the form $(b',b+1)$ and $(b',b)$ such that $(b,b+1)$ is not a leaf of $\calT_{F'}$, by \Cref{rmk-key-facts} \ref{rmk-key-facts-6}. Since $b'<b$, we also have $(b',b+1),(b',b)\in F$ by the minimality of $b$. Meanwhile, as $\calT_F$ and $\calT_{F'}$ have a common leaf set, $(b,b+1)$ is not a leaf of $\calT_{F}$ as well. Whence, $i_{b,1}>b+1$ and we have two crossing open intervals $(b',b+1)$ and $(b,i_{b,1})$ for $F$, which is another contradiction.

        \item \label{no-higher-deg-2} When $k_b=2$, we observe first that $(b,i_{b,1})\in F\cap F'$. To see this, notice that if $(b,i_{b,1})\in F\setminus F'$, then $(b,i_{b,1})$ is not a common leaf. Whence, $(b,i_{b,1})\in LG_F$ by our description in subsection \ref{Linear generators}. Now, we have $(b,i_{b,1})\in LG_F
            \cap (F\setminus F')$, contradicting the assumption in \eqref{eqn:high-deg}.

            Therefore, by the minimality of $b$ and the fact that $\widehat{\bdT^{F'}} >_{\lex} \widehat{\bdT^{F}}$, we have two subcases: either $(b,i')\in F_b'$ for some $i'>i_{b,2}$ or $F_b'=\{(b,i_{b,1})\}$.  One can likewise exclude the first subcase using the non-crossing argument. Whence, we only need to check with the  $F_b'=\{(b,i_{b,1})\}$ case. In $\calT_{F'}$, the parent node of $(b,i_{b,1})$ must take the form of $(b',i_{b,1})$ for some $b'<b$. By the minimally of $b$, $(b',i_{b,1})$ belongs to $F\cap F'$ and will consequently be the parent node of $(b,i_{b,1})$ in both $\calT_F$ and $\calT_{F'}$. But the parent node of $(b,i_{b,1})$ in $\calT_F$ is $(b,i_{b,2})$, still a contradiction.

        \item Assume now that $k_b\ge 3$. If $(b,i_{b,j})\in F_b'$ for all $1\le j\le k_b-1$, then we can argue as in the previous case \ref{no-higher-deg-2}. Thus, we have some $j$ with $1\le j\le k_b-1$ such that $(b,i_{b,j})\notin F_b'$. Let $j_0$ be the smallest with this property. Notice that
            \begin{itemize}
                \item we always have $(b,i_{b,1})\in F'$ as in \ref{no-higher-deg-2}, and
                \item for $j$ with $2\le j\le k_b-1$, if $(b,i_{b,j})$ has a sibling or two children, then $(b,i_{b,j})\in LG_F\subset F'$ by \eqref{eqn:high-deg}.
            \end{itemize}
            Thus, $2\le j_0\le k_b-1$ and $(b,i_{b,j_0})$ has neither a right child nor a right sibling in $\calT_F$.  

            As a matter of fact, $(b,i_{b,j})$ has no right sibling in $\calT_F$ for $j_0\le j\le k_b-1$.  To see this, suppose for contradiction that there exists some $j_1>j_0$ such that $(b,i_{b,j_1})$ has a right sibling. Let $j_1$ be the smallest. Therefore,  for all $j$ with $j_0\le j< j_1$, $(b,i_{b,j})$ has no right sibling. Notice that $(b,i_{b,j_0-1})$ also has no right sibling. Therefore, $(i_{b,j_0-1},i_{b,j_1})$ contains no leaf of $\calT_F$ and $i_{b,j}=i_{b,j_0-1}+j-(j_{0}-1)$  for $j_0\le j\le j_1$ by \Cref{rmk-key-facts} \ref{rmk-key-facts-2}. 
            Since $F,F'\in \calG_{\alpha}$, it follows that $(i_{b,j_0-1},i_{b,j_1})$ contains no leaf of $\calT_{F'}$.
            Meanwhile, as $(b,i_{b,j_1})$ has a right sibling in $\calT_F$, $(b,i_{b,j_1})\in LG_F\subset F'$ as well. Since we have assumed that $(b,i_{b,j_0-1})\in F'$, this forces $(b,i_{b,j})\in F'$ for $j_0-1<j\le j_1$ by \Cref{rmk-key-facts} \ref{rmk-key-facts-4}, contradicting the choice of $j_0$. 

            Now, $(b,i_{b,j})$ has no right sibling in $\calT_F$ for $j_0-1\le j\le k_b-1$. It follows from 
            \Cref{rmk-key-facts} \ref{rmk-key-facts-2} that
            the interval $(i_{b,j_0-1},i_{b,k_b})$ has no leaf of $\calT_F$, and
            \begin{equation}
                i_{b,j}=i_{b,j_0-1}+j-(j_0-1)\qquad \text{for}\quad j_0-1\le j\le k_b.
                \label{eqn:b-no-tail-eqn}
            \end{equation}
           
            Furthermore, as $(b,i_{b,j_0-1})\in F'$ while $(b,i_{b,j_0})\notin F'$, we have 
            \begin{equation}
                (b,i_{b,j})\in F\setminus F'\qquad \text{for}\quad j_0\le j\le k_b
                \label{eqn:b-no-tail}
            \end{equation}
            by \Cref{rmk-key-facts} \ref{rmk-key-facts-4}.

            So far, $(b,i_{b,1})\in F\cap F'$ and $(b,i_{b,j})\in F_b'$ for all $j=2,3,\dots,j_0-1$ by the minimality of $j_0$ . Therefore, we have
            \begin{equation}
                \Set{(b,i_{b,j})|1\le j\le j_0-1} \subseteq F_b'.
                \label{eqn:contain-equal}
            \end{equation}
            If the containment in \eqref{eqn:contain-equal} is strict, then we have some $(b,i')\in F_b'\setminus F_b$. Since $F'\succ F$, $i'>i_{b,j_0}$. And by \eqref{eqn:b-no-tail-eqn} with \eqref{eqn:b-no-tail}, we must have $i'>i_{b,k_b}$. If $b=1$, then $(b,i_{b,k_b})$ is the common root $(1,c)$. The existence of such an $i'$ is impossible. Therefore, $b>1$ and the parent node of $(b,i_{b,k_b})$ in $\calT_F$ must have the form $(b',i_{b,k_b})$ for some $b'<b$.
            By the minimality of $b$, $(b',i_{b,k_b})\in F'$.
            Whence, we will arrive at a contradiction due to crossing. Therefore, the containment in \eqref{eqn:contain-equal} is actually an equality:
            \[
                \Set{(b,i_{b,j})|1\le j\le j_0-1} = F_b'.
            \]
            Now, we consider the parent node of $(b,i_{k,j_0-1})$ in $\calT_{F'}$ and get a similar contradiction due to crossing. And this completes the proof in this subsection.
    \end{enumerate}

    In summary, the ideal $(\ini(P))^{\vee}$ has linear quotients, as expected.
\end{proof}

\begin{Corollary}
    When $c\ge d+4$, the initial complex $\Delta$ is shellable.
\end{Corollary}

\begin{proof}
    It follows from \cite[Proposition 8.2.5]{MR2724673} and \Cref{prop:linear-quotients}.
\end{proof}

%\section{Regularity of fiber cones of rational normal scrolls}
\section{Applications of the Cohen--Macaulay property}
% \label{reg-scroll}

In this final section, $I_{n_1,\dots,n_d}$ is still the defining ideal of the rational normal scroll $\calS_{n_1,\dots,n_d}\subseteq \PP^{c+d-1}$ for $c=n_1+\cdots+n_d$. 

As the first application to our main result \Cref{thm:CM-fiber-cone}, 
%As the main result of this section, 
we compute the Castelnuovo--Mumford regularity of the fiber cone $\calF(I_{n_1,\dots,n_d})$ in \Cref{thm:reg-2}. Its argument depends on the Cohen--Macaulayness established in section 2 and the regularity result of balanced cases in \cite{Lin-Shen3}. Again, by the Cohen--Macaulayness, the regularity of the fiber cone can be related to its {$\bda$-invariant}, as well as the reduction number of the ideal $I_{n_1,\dots,n_d}$.

%As a quick reminder, 
It is worth mentioning that, 
\cite[Section 20.5]{MR1322960} provides a concise introduction to the notion of \emph{Castelnuovo--Mumford regularity}, as well as some historical notes. And the \emph{$\bda$-invariant} was introduced by Goto and Watanabe in \cite[Definition 3.1.4]{MR494707}. A nice thing is that, for any standard graded Cohen--Macaulay algebra $A$ over $\KK$, one has
\begin{equation}
    \bda(A)=\reg(A)-\dim(A),
    \label{def:a-inv}
\end{equation}
in view of the equivalent definition of regularity in \cite[Definitions 1 and 3]{MR676563}. As for the \emph{reduction number} of an ideal and some related topics, one may wish to consult \cite[Section 8.2]{MR2266432}.  Since these three invariants have been widely studied in the commutative algebra literature, we will not bother repeating their definitions here. 
    
%We need the following observation. We end this section with a quick application. It is also due to the following fact.
%However, before stating the main result of this section, 
However, before stating the concrete result and displaying its proof,
we still need to include three important facts here for completeness.

\begin{Lemma}
    [{\cite[Proposition 7.43]{MR2508056}}]
    \label{lem:reg-deg}
    Let $M$ be a graded Cohen--Macaulay module over the polynomial ring $\KK[x_1,\dots,x_n]$, 
    and $P_M(t)$ the numerator Laurent polynomial of the Hilbert series of $M$. Then $\reg(M)=\deg(P_M(t))$.
\end{Lemma}

\begin{Lemma}
    [{\cite[Proposition 6.6]{CNPY} or \cite[Proposition 1.2]{MR3864202}}]
    \label{CNPY:6.6}
    Let $I \subset R = \KK[x_1,\dots,x_N]$ be a homogeneous ideal that is generated in one degree, say $d$. Assume that the fiber cone $\calF(I)$ is Cohen--Macaulay. Then each minimal reduction of $I$ is generated by $\dim(\calF(I))$ homogeneous polynomials of degree $d$, and $I$ has the reduction number $\r(I) = \reg(\calF(I))$.
\end{Lemma}

Recall that the rational normal scroll $\calS_{n_1,\dots,n_d}$ is \emph{balanced} precisely when $|n_{i}-n_{j}|\le 1$ for all $i,j$. Whence, we can rearrange the columns of the matrix $\bfX$ in the equation \eqref{eqn:matrix-X} and rewrite it as the \emph{Hankel matrix} $\bdH_{2,c,d}$ for $c=n_1+\cdots+n_d$. And here is the last piece that we need for the regularity result.

\begin{Lemma}
    [{\cite[Theorem 3.1]{Lin-Shen3}}]
    \label{thm:main-reg-balanced}
    Let $I_2(\bdH_{2,c,d})$ be the defining ideal of the rational normal scroll $\calS_{n_1,\dots,n_d}\subseteq \PP^{c+d-1}$ in the balanced case. Then, the Castelnuovo--Mumford regularity of the fiber cone $\calF(I_2(\bdH_{2,c,d}))$ is
    \[
        \reg(\calF(I_2(\bdH_{2,c,d})))=
        \begin{cases}
            \ceil{(c+d-1)/2}, & \text{if $4+d\le c$},\\
            c-3, & \text{if $2< c \le 4+d$}.
        \end{cases}
    \]
\end{Lemma}

It is time to describe the first result of this section.

\begin{Theorem}
    \label{thm:reg-2}
    Let $I_{n_1,\dots,n_d}$ be the defining ideal of the rational normal scroll $\calS_{n_1,\dots,n_d}\subseteq \PP^{c+d-1}$ for $c=n_1+\cdots+n_d$. 
    \begin{enumerate}[i]
        \item \label{thm:reg-2-1}
        The Castelnuovo--Mumford regularity of the fiber cone is given by
    \[
        \reg(\calF(I_{n_1,\dots,n_d}))=
        \begin{cases}
            \ceil{(c+d-1)/2}, & \text{if $4+d \le c$},\\
            c-3,& \text{if $2< c<4+d$}.
        \end{cases}
    \]
    \item \label{thm:reg-2-2}
    The reduction number of $I_{n_1,\dots,n_d}$ is given by
    \[
        \r(I_{n_1,\dots,n_d})=
        \begin{cases}
            \ceil{(c+d-1)/2}, & \text{if $4+d \le c$},\\
            c-3,& \text{if $2< c<4+d$}.
        \end{cases}
    \]
    \item \label{thm:reg-2-3}
    The $\bda$-invariant of the fiber cone is given by
     \[
        \bda(\calF(I_{n_1,\dots,n_d}))=
        \begin{cases}
            \ceil{(c+d-1)/2}-c-d, & \text{if $4+d \le c$},\\
            -c,& \text{if $2< c<4+d$}.
        \end{cases}
    \]   
    \end{enumerate}
\end{Theorem}

\begin{proof}
    Notice that the fiber cone $\calF(I_{n_1,\dots,n_d})$ is Cohen--Macaulay by \Cref{thm:CM-fiber-cone}. Thus, the item \ref{thm:reg-2-2} follows directly from the item \ref{thm:reg-2-1} and \Cref{CNPY:6.6}. And similarly, the item \ref{thm:reg-2-3} follows from the item \ref{thm:reg-2-1}, Equation \eqref{def:a-inv} and the fact
    \[
        \dim(\calF(I_{n_1,\dots,n_d}))=
        \begin{cases}
            c+d, & \text{if $4+d \le c$,}\\
            2c-3,& \text{if $2< c<4+d$,}
        \end{cases}
    \]
    by \cite[Proposition 3.1, Corollary 3.10]{MR4068250}.
    
    As for the item \ref{thm:reg-2-1}, just as observed in the proof of \cite[Theorem 3.13]{MR4068250}, the Hilbert function of the fiber cone of the rational normal scroll $\calS_{n_1,\dots,n_d}$ depends only on $c$ and $d$ by \cite[Theorem 3.7]{MR3341465}.
    Whence, the fiber cones of the ideal $I_{n_1,\dots,n_d}$ above and the ideal $I_2(\bdH_{2,c,d})$ in \cite[Section 3]{Lin-Shen3} have the same Laurent polynomial. Now, by the Cohen--Macaulayness of $\calF(I_{n_1,\dots,n_d})$ established in \Cref{thm:CM-fiber-cone}, these fiber cones have the same Castelnuovo--Mumford regularity by \Cref{lem:reg-deg}. 
    
    Now, it remains to apply \Cref{thm:main-reg-balanced}.
    %\cite[Theorem 4.1]{Lin-Shen3}, where we computed the regularity of the fiber cones of the rational normal scrolls and their secant varieties in the \emph{balanced case}.  
\end{proof}

As the second application to our main result \Cref{thm:CM-fiber-cone}, we characterize the Gorensteinness of the fiber cone. The corresponding numerical condition in the balanced case is the following.

\begin{Lemma}
    [{\cite[Proposition 2.10]{Lin-Shen3}}]
    \label{lem:Gor}
    Let $I_2(\bdH_{2,c,d})$ be the defining ideal of the rational normal scroll $\calS_{n_1,\dots,n_d} \subseteq \PP^{c+d-1}$ in the balanced case. Then, the fiber cone $\calF(I_2(\bdH_{2,c,d}))$ is Gorenstein if and only if $c\in \Set{2,3,2+d,3+d,4+d}$.
\end{Lemma}

And here is the last result that this paper wants to describe.

\begin{Theorem}
    \label{thm:Gor}
    Let $I_{n_1,\dots,n_d}$ be the defining ideal of the rational normal scroll $\calS_{n_1,\dots,n_d}\subseteq \PP^{c+d-1}$ for $c=n_1+\cdots+n_d$. Then, the fiber cone $\calF(I_{n_1,\dots,n_d})$ is Gorenstein if and only if $c\in \Set{2,3,2+d,3+d,4+d}$.
\end{Theorem}

\begin{proof}
    The fiber cone $\calF(I_{n_1,\dots,n_d})$ is Cohen--Macaulay by \Cref{thm:CM-fiber-cone}. And since it is isomorphic to the toric ring $\KK[I_{n_1,\dots,n_d}]$, it is an integral domain.
    As mentioned in the proof of \Cref{thm:reg-2}, the fiber cones $\calF(I_{n_1,\dots,n_d})$ and $\calF(I_2(\bdH_{2,c,d}))$ have identical Hilbert series. Thus, by both (b) and (c) of \cite[Corollary 4.4.6]{MR1251956}, $\calF(I_{n_1,\dots,n_d})$ is Gorenstein, if and only $\calF(I_2(\bdH_{2,c,d}))$ is so. Now, it suffices to apply \Cref{lem:Gor}.
\end{proof}

%\begin{Remark}
%    For the ideal $I=I_2(\bdH_{2,c,d})$ stated above, when $c=2+d$, $3+d$ or $4+d$, the paper \cite{arXiv:1901.01561} of Dinu showed that the corresponding fiber cone $\calF(\ini(I))$ is Gorenstein. And in the same paper, he also computed some geometric invariant $\delta$ of the associated integral convex polytope. Notice that in these Gorenstein cases, this integer $\delta$ is simply $-\bda(\calF(\ini(I)))$ by \cite[Proposition 2.2]{Noma}. Whence, one can also get the regularity for free.
%\end{Remark}

\begin{acknowledgment*}
    The authors are grateful to the software system \texttt{Macaulay2} \cite{M2}, for serving as an excellent source of inspiration.
    The second author is partially supported by the ``Anhui Initiative in Quantum Information Technologies'' (No.~AHY150200) and the ``Fundamental Research Funds for the Central Universities''.
\end{acknowledgment*}

\bibliography{FiberCone}

\end{document}